\documentclass[11pt]{article}
\usepackage{amsmath,amsthm}
\usepackage{amsfonts,mathrsfs}
\usepackage{color}
\usepackage[colorlinks,anchorcolor=blue,citecolor=blue]{hyperref} 
\usepackage{enumerate} 
\usepackage[round]{natbib}

\usepackage{multicol}
\usepackage[margin=1in]{geometry}

\allowdisplaybreaks

\newtheorem{thm}{Theorem}[section]
\newtheorem{prop}[thm]{Proposition}

\newtheorem{rmk}[thm]{Remark}

\renewcommand{\P}{\mathbb{P}}
\newcommand{\D}{\displaystyle}
\newcommand{\non}{\nonumber}
\newcommand{\E}{\mathbb{E}}
\newcommand{\RR}{\mathbb{R}}

\newcommand{\mo}{\mathcal{O}^{(p,q)}}
\newcommand{\mop}{\mathcal{O}^{+}}
\newcommand{\mom}{\mathcal{O}^{-}}

\newcommand{\wpq}{W^{(p,p+q)}}

\numberwithin{equation}{section}

\usepackage[pagewise]{lineno}
\begin{document}

\title{An excursion theoretic approach to Parisian ruin problem}

\vskip 1cm

\author{Bo Li and Xiaowen Zhou}

\maketitle

\centerline{\Large Abstract}
\vskip 0.5cm
Applying excursion theory, we re-express several well studied fluctuation quantities associated to Parisian ruin problem for L\'evy risk processes in terms of integrals with respect to excursion measure for spectrally negative L\'evy process.
We show that these new expressions reconcile with the previous results on Parisian ruin problem.

\section{Introduction}

Parisian time was first proposed in \cite{Ch97} for option pricing in mathematical finance, and was later introduced to risk theory in \cite{DaWu99a} and in \cite{DaWu99b} for risk models with implementation delay. More precisely, Parisian ruin occurs once the surplus process has continuously stayed below level $0$ for a time interval of length $d$, i.e. an excursion below $0$ of the surplus process has first reached length $d$. \cite{Loeffen2013} first investigated the Parisian ruin problem for spectrally negative L\'evy risk processes and obtained an expression for the Parisian ruin probability. Since then there have been a long list of papers on Parisian ruin related problems for various risk models.

Although the Parisian ruin is defined using excursions for the surplus process and the excursion theory for spectrally negative L\'evy processes had been well studied, surprisingly, in the previous work, excursion theory was rarely applied directly to derive expressions for the Parisian ruin probability or Laplace transform for the Parisian ruin time. Instead, some approximating-limiting arguments are often implemented except in \cite{WZ20} where the excursion theory was adopted to obtain expressions for quantities related to the so called draw-down Parisian ruin.

The main purpose of this paper is to introduce a new systematic excursion theoretic approach to the study of Parisian ruin problems for L\'evy risk processes. We believe this alternative approach will lead to new advances in the study of Parisian ruin related problems for L\'evy risk processes. To this end, we first find joint Laplace transforms (in terms of scale functions) on the positive and negative durations up to certain exit times for a generic excursion under excursion measure, which relies on Laplace transforms of weighted occupation times obtained in \cite{LiZ2019}. 
Using excursion theory we then express Parisian ruin related quantities such as the Parisian ruin probability, Laplace transform of the Parisian ruin time and potential measure up to the Parisian ruin time in terms of integrals with respect to the excursion measure. In particular, we can recover the nowly well known results in \cite{Loeffen2013} and \cite{Loeffen2018}.

The rest of the paper is arranged as follows. In Section \ref{Pre} we present the Parisian ruin for L\'evy risk processes together with some fluctuation results and introduce the excursion theory for spectrally negative L\'evy processes. The main results on new expressions for the above-mentioned Parisian ruin related quantities are obtained in Section \ref{Main} applying the excursion theory. The previous results of \cite{Loeffen2013} and \cite{Loeffen2018} are recovered in Section \ref{recover}.

\section{Preliminaries}\label{Pre}
\subsection{L\'evy insurance surplus processes}
A process $X$ is a surplus process for a L\'evy insurance risk model if $X$ is a spectrally negative L\'evy process (SNLP for short) defined on a filtered probability space $\big(\Omega,\mathcal{F},(\mathcal{F}_{t})_{t\ge0},\P\big)$, that is, a stochastic process with independent stationary increments and without positive jumps.
For $X_0=0$, its Laplace exponent exists and takes the L\'evy-Khintchine form
\begin{equation}\label{defn:psi}
\psi(\theta):=\frac{1}{t}\log\E\big[e^{\theta X(t)}\big]=\gamma\theta+\frac{\sigma^{2}}{2}\theta^{2}+\int_{0}^{\infty}\big(e^{-\theta z}-1+\theta z\mathbf{1}(z\le1)\big)\nu(dz),\,\, \theta\geq 0,
\end{equation}
where $\gamma\in\mathbb{R},\sigma\ge0$ are constants and $\nu$ is a $\sigma$-finite measure on $(0,\infty)$ with $\int_{0}^{\infty}(1\wedge z^{2})\nu(dz)<\infty$.
Function $\psi$ in \eqref{defn:psi} is continuous and strictly convex on $[0,\infty)$, whose right inverse is defined by $\phi(s):=\sup\{\theta\ge0: \psi(\theta)=s\}$ for $s\ge0$.

We only consider the process $X$ that is not monotone.
Define the upward and downward first passage times for process $X$, respectively, by
\[
\tau_{x}^{+}:=\inf\{t>0: X(t)>x\}
\quad\text{and}\quad
\tau_{x}^{-}:=\inf\{t>0: X(t)<x\},
\]
with the convention $\inf\emptyset=\infty$.
The first passage problems of $X$ have been well studied via the fluctuation theory of SNLP,
and the following scale functions $(W^{(q)})_{q\ge0}$ plays a key role,
which for $q\geq 0$ is defined as the unique continuous and increasing function on $[0,\infty)$ satisfying
\begin{equation}
\label{defn:wq}
\int_{0}^{\infty}e^{-sy}W^{(q)}(y)\,dy=\frac{1}{\psi(s)-q}\quad\text{for $s>\phi(q)$},
\end{equation}
with extension $W^{(q)}(x)=0$ for $x<0$. We write $W(x)\equiv W^{(0)}(x)$ for $q=0$ and refer the readers to \cite{Bertoin96book} and \cite{Kyprianou2014book} for more detailed discussions on SNLP and its scale functions. The second scale function is defined as for $x\in\mathbb{R}$
\begin{equation}
\label{defn:zq}
Z^{(q)}(x):=1+q\int_{0}^{x}W^{(q)}(z)\,dz=1+ q\int_{-\infty}^{x}W^{(q)}(z)\,dz.
\end{equation}
The following limits can be found in \cite[Exe. 8.5]{Kyprianou2014book} for $q>0$
\begin{equation}
\label{wz-limit}
e^{-\phi(q)x}W^{(q)}(x)\uparrow\phi'(q)<\infty
\quad\text{and}\quad
\frac{Z^{(q)}(x)}{W^{(q)}(x)}\to\frac{q}{\phi(q)}
\quad\text{as $x\to\infty$}
\end{equation}
where the limits also hold for the case $q=0$ with $\psi'(0)\neq0$.
If $\psi'(0)=0$, then $\phi(0)=0$ and we have
\[W(x)\to\infty,\quad
\frac{Z(x)}{W(x)}=\frac{1}{W(x)}\to0
\quad\text{and}\quad
\frac{W(x-y)}{W(x)}\to1.\]

Define, respectively, the occupation time above and below $0$ for $X$ by
\[\mop(t):=\int_{0}^{t}\mathbf{1}(X(s)>0)\,ds
\quad\text{and}\quad\mom(t):=\int_{0}^{t}\mathbf{1}(X(s)<0)\,ds,
\]
and for any $p,q>0$, define the weighted occupation time as
\[
\mo(t):=\int_{0}^{t}\big(p\mathbf{1}(X(s)>0)+q\mathbf{1}(X(s)<0)\big)\,ds=p\mop(t)+q\mom(t).
\]
The following results on $\mo$ were proved in \cite{Loeffen2014} and \cite{LiYZ2014},
except that on the occupation time at inverse local time which was proved in \cite{LiZ2019}.
\begin{prop}\label{prop:pq:ab}
Given $b>0>c$ and $p,q>0$,
we have for any $x\in(c,b)$
\begin{gather}
\E_{x}\Big[\exp\Big(-\big(p\mop+q\mom\big)(\tau_{b}^{+})\Big); \tau_{b}^{+}\le\tau_{c}^{-}\Big]
=\frac{W^{(p,q)}(x,c)}{W^{(p,q)}(b,c)},\label{lapocc:b<c}\\
\E_{x}\Big[\exp\Big(-\big(p\mop+q\mom\big)(\tau_{c}^{-})\Big); \tau_{c}^{-}\le\tau_{b}^{+}\Big]
=Z^{(p,q)}(x,c)-\frac{W^{(p,q)}(x,c)}{W^{(p,q)}(b,c)}Z^{(p,q)}(b,c),\label{lapocc:c<b}
\end{gather}
where $W^{(p,q)},Z^{(p,q)}$ are functions on $\RR^{2}$ given by
\begin{equation}\label{defn:wzpq}
\begin{aligned}
W^{(p,q)}(x,y)
:=&\ W^{(q)}(x-y)+(p-q)\int_{0}^{x}W^{(p)}(x-z)W^{(q)}(z-y)\,dz,\\
Z^{(p,q)}(x,y)
:=&\ Z^{(q)}(x-y)+(p-q)\int_{0}^{x}W^{(p)}(x-z)Z^{(q)}(z-y)\,dz.
\end{aligned}
\end{equation}
For the weighted occupation time at inverse local time, we have
\begin{equation}
\E\Big[\exp\Big(-\mo(L^{-1}(t))\Big); L^{-1}(t)<\tau_{b}^{+}\wedge\tau_{c}^{-}\Big]
=\exp\Big(\frac{-W^{(p,q)}(b,c)\cdot t}{W^{(p)}(b)W^{(q)}(-c)}\Big).\label{lapocc:bc}
\end{equation}
Moreover, for any bounded and measurable function $f$ on $\mathbb{R}$, we have 
\begin{equation}\label{lapocc:resl}
\E_{x}\bigg[\int_{0}^{\tau_{b}^{+}\wedge\tau_{c}^{-}}e^{-\mo(t)} f\big(X(t)\big)\,dt\bigg]
=\int_{c}^{b}f(y)\Big(\frac{W^{(p,q)}(x,c)}{W^{(p,q)}(b,c)}W^{(p,q)}(b,y)-W^{(q)}(x,y)\Big)\,dy.
\end{equation}
\end{prop}
\begin{rmk}
The generalized scale functions
$W^{(p,q)}$ and $Z^{(p,q)}$ were first introduced in \cite{Loeffen2014} and further generalized in \cite{LiP2018} for the study of weighted occupation times of SNLP.
It follows from \eqref{defn:wzpq} that
\[
W^{(p,q)}(x,y)=0
\,\,\text{and}\,\,
Z^{(p,q)}(x,y)=1\,\,\text{ for}\,\, x<y\]
and for $x>0>y$
$$Z^{(p,q)}(x,0)=Z^{(p)}(x), \,\, Z^{(p,q)}(0,y)=Z^{(q)}(-y)\,\,\text{and}\,\, Z^{(p,0)}(x,y)=Z^{(p)}(x),$$
see \cite{LiZ2019} for more properties of the functions.
\end{rmk}

\subsection{L\'evy surplus processes with Parisian ruin}\label{sec:parisianSNLP}
The Parisian ruin time with duration $\gamma>0$ for a surplus process $X$ is defined as
\[
\tau_{\gamma}:=\inf\{t>\gamma: X(t-s)<0\,\, \text{for all}\,\, s\in[0,\gamma]\},
\]
which is the first moment when the process has continuously stayed below $0$ for a period of length $\gamma$.
Making use of Kendall's identity, \cite{Loeffen2013,Loeffen2018} found nice compact expressions for the Parisian problem for SNLP $X$
by introducing the following auxiliary function
\[
\Lambda^{(p)}(x,t):=\int_{0}^{\infty}W^{(p)}(x+z)\frac{z}{t}\P\big(X(t)\in dz\big),
\]
which satisfies for $s>0$ and $x\in\mathbb{R}$,
\begin{equation}
\label{ken:w}
\int_{0}^{\infty}e^{-st}e^{-pt}\Lambda^{(p)}(x,t)\,dt=\int_{0}^{\infty}e^{-\phi(p+s)z}W^{(p)}(z+x)\,dz.
\end{equation}

The following Laplace transforms and potential measures associated to the Parisian ruin time were obtained in \cite[Theorem 3.1, Corollary 3.2 and Proposition 3.4]{Loeffen2018}.
\begin{prop}\label{prop:ro18}
For $b,p>0$ and $x\le b$ we have
\begin{gather}
\E_{x}\Big[e^{-p\tau_{b}^{+}}; \tau_{b}^{+}<\tau_{\gamma}\Big]=\frac{\Lambda^{(p)}(x,\gamma)}{\Lambda^{(p)}(b,\gamma)},\label{prop:ro18:b<r}\\
\begin{aligned}
\E_{x}\Big[e^{-p(\tau_{\gamma}-\gamma)}; \tau_{\gamma}<\tau_{b}^{+}\Big]
=&\ \Big(Z^{(p)}(x)+p \int_{0}^{\gamma}\Lambda^{(p)}(x,s)\,ds\Big)\\
&\quad -\frac{\Lambda^{(p)}(x,\gamma)}{\Lambda^{(p)}(b,\gamma)}\Big(Z^{(p)}(b)+ p\int_{0}^{\gamma}\Lambda^{(p)}(b,s)\,ds\Big).
\end{aligned}\label{prop:ro18:r<b}
\end{gather}
Moreover, for $y\ge0$ we have
\begin{equation}\label{prop:ro18:resl}
\int_{0}^{\infty}e^{-pt}\P_{x}\big(X_{t}\in dy, t<\tau_{b}^{+}\wedge\tau_{\gamma}\big)\,dt
=\Big(\frac{\Lambda^{(p)}(x,\gamma)}{\Lambda^{(p)}(b,\gamma)}W^{(p)}(b-y)-W^{(p)}(x-y)\Big)\,dy.
\end{equation}
\end{prop}

\begin{rmk}\label{rmk_markov}
The joint Laplace transform of $X$ at $\tau_{\gamma}$ was also found in \cite[Theorem 3.1]{Loeffen2018}, which is essentially equivalent to \eqref{prop:ro18:r<b} by change of measure. So is the potential measure in \cite[Proposition 3.4]{Loeffen2018},
where expression for the potential measure is only found for $y\in (0,\infty)$. 
Note that $X(\cdot\wedge\tau_{\gamma})$ does not have the Markov property, whereas the two-dimensional process $(X,l)(\cdot\wedge\tau_{\gamma})$ does, where $l(t)$ denotes the length of negative duration of $X$ right before $t$, i.e.
\[
l(t):=\sup\{r>0: X(t-s)<0, \forall s\in[0,r]\},
\]
and $l(t):=0$ if $X(t)\ge0$. Therefore, it would be more convenient to state the results in terms of the two-dimensional process, and in particular, for $(x,t)$ with $x<0$ and $t>0$.
\end{rmk}

\subsection{Some excursion theory for SNLPs}
Different from the limiting arguments used in \cite{Loeffen2013,Loeffen2018} among others, in this paper
we employ the excursion theory for L\'evy processes where the \textit{compensation formula} for Poisson point process is frequently used; see e.g. \cite[Chapters O.5 and IV]{Bertoin96book}.
The excursion theory is at full strength if $0$ is regular for $X$, which is equivalent to the assumption that $X$ has sample paths of unbounded variation,
and is further equivalent to the condition that 
\begin{center}
either $\sigma>0$ or $\D\int_{0}^{1}x\Pi(dx)=\infty$.
\end{center}
Therefore, we always assume in this paper that $X$ has sample paths of unbounded variation,
and recall some notion and basic facts about excursion theory in this section.

Let $\mathcal{L}:=\{t>0: X(t)=0\}$ be the zero set of $X$, and $\bar{\mathcal{L}}$ be its closure.
Given that $0$ is regular for $X$, the local time process of $X$ at $0$ exists and can be defined as the occupation density process
\[
L(t):=\lim_{\varepsilon\to0+}\frac{1}{2\varepsilon}\int_{0}^{t}\mathbf{1}\big(| X(s)|<\varepsilon\big)\,ds, \,\, t\geq 0,
\]
which is a continuous additive functional that increases only on $\bar{\mathcal{L}}$,
see e.g. \cite[VI.1]{Bertoin96book}.
The time spent by $X$ at $0$ is a Lebesgue-null set, i.e. 
\begin{equation}\label{d=0}
\int_{0}^{t}\mathbf{1}\big(s\in\bar{\mathcal{L}}\big)\,ds=0,\quad\forall t>0.
\end{equation}
For $t>0$ define the inverse of $L$ by
\[
L^{-1}(t):=\inf\{s>0: L(s)>t\}\quad\text{and}\quad
L^{-1}(t-):=\inf\{s>0: L(s)\ge t\}.
\]

The excursion process of $X$ away from $0$ is the pieces of sample path of $X$,
which takes values in the space $\mathscr{E}$ of excursion functions together with an additional isolated point $\Delta$,
and is defined as
\[
\epsilon_{t}:=\{X(L^{-1}(t-)+s), s\in[0, L^{-1}(t)-L^{-1}(t-))\}=:\{\epsilon_{t}(s), s\in[0,\zeta_{t})\}
\]
for the local time index $t$ satisfying $L^{-1}(t-)<L^{-1}(t)$ and $\epsilon_{t}=\Delta$ otherwise,
where $\mathscr{E}$ is a collection of c\`adl\`ag trajectories stopped at its first time hitting $0$,
and $\zeta_{t}:=L^{-1}(t)-L^{-1}(t-)$ represents the lifetime of $\epsilon_{t}$.
It is well known that there are countably many $t>0$ such that $L^{-1}(t-)<L^{-1}(t)$,
and $\{\epsilon_{t}, 0\leq t\le L(\infty)\}$ is a Poisson point process (stopped at the first excursion with infinite lifetime),
which is characterised by the associated excursion measure $n(\cdot)$; see e.g. \cite[IV.10]{Bertoin96book}.


Given a generic excursion $\epsilon\equiv\{\epsilon(s), s\in [0,\zeta) \}\in\mathscr{E}$, its total amount of time above $0$ and below $0$ are defined, respectively, by
\[
\zeta^{+}:=\int_{0}^{\zeta}\mathbf{1}(\epsilon(s)>0)\,ds
\quad\text{and}\quad
\zeta^{-}:=\int_{0}^{\zeta}\mathbf{1}(\epsilon(s)<0)\,ds,
\]
 and its first passage times are defined by
\[
\kappa_{x}^{+}:=\inf\{s<\zeta, \epsilon(s)>x\}
\quad\text{and}\quad
\kappa_{x}^{-}:=\inf\{s<\zeta, \epsilon(s)<x\}
\]
with the convention $\inf\emptyset=\infty$.
We use subscripts $t$, for example, $\kappa_{c,t}^{-}$ and $\kappa_{b,t}^{+}$ to emphasize the association with excursion $\epsilon_{t}\in\mathscr{E}$.

The following excursion theory results, which could also be derived from \cite{Pardo2018}, are proved in appendix for completeness.
\begin{prop}\label{prop:npq:bc}
For $b>0>c$ and $p,q>0$, we have
\begin{gather}
n\big(1-e^{-p\zeta^{+}-q\zeta^{-}}\mathbf{1}\big(\zeta<\kappa_{b}^{+}\wedge\kappa_{c}^{-}\big)\big)
= \frac{W^{(p,q)}(b,c)}{W^{(p)}(b)W^{(q)}(-c)},\label{n:pq:bc}\\
n\big(e^{-p\kappa_{b}^{+}}; \kappa_{b}^{+}<\zeta\big)
=\frac{1}{W^{(p)}(b)},\label{n:pq:b<c}\\
n\big(e^{-p\zeta^{+}-q(\kappa_{c}^{-}-\zeta^{+})}; \kappa_{c}^{-}<\zeta\wedge\kappa_{b}^{+}\big)
= \frac{Z^{(q)}(-c)W^{(p,q)}(b,c)-Z^{(p,q)}(b,c)W^{(q)}(-c)}{W^{(p)}(b)W^{(q)}(-c)}.\label{n:pq:c<b}
\end{gather}
Define for $t>0$
\[
\eta^{(p,q)}(\epsilon,t):=\int_{0}^{t\wedge\zeta}\big(p\mathbf{1}(\epsilon(s)>0)+q\mathbf{1}(\epsilon(s)<0)\big)\,ds.
\]
as the weighted occupation time up to $t$ for excursion $\epsilon$.
Then for any bounded measurable function $f$ on $\mathbb{R}$, we have
\begin{equation}\label{n:pq:resl}
\begin{aligned}
n\Big(\int_{0}^{\zeta\wedge\kappa_{b}^{+}\wedge\kappa_{c}^{-}}e^{-\eta^{(p,q)}(\epsilon,t)}f(\epsilon(t))\,dt\Big)
=\int_{c}^{b}f(y)\Big(\frac{W^{(p,q)}(b,y)}{W^{(p)}(b)}-\frac{W^{(p,q)}(b,c)W^{(q)}(-y)}{W^{(p)}(b)W^{(q)}(-c)}\Big)\,dy.
\end{aligned}
\end{equation}
\end{prop}
\begin{rmk}
 Observe from \eqref{n:pq:bc} that, on the set $\{\zeta<\infty\}$ and under the excursion measure $n$, the joint length-height distribution of the positive part of the excursion has the same law as the length-depth distribution of the negative part, i.e.
\[\begin{aligned}
&\ n\big(\zeta^{+}\in dt,\zeta^{-}\in ds, \sup_{r<\zeta}\epsilon(r)\in du, \big|\inf_{r<\zeta}\epsilon(r)\big| \in dv\big|\zeta<\infty\big)\\
=&\ n\big(\zeta^{+}\in ds,\zeta^{-}\in dt, \sup_{r<\zeta}\epsilon(r)\in dv, \big|\inf_{r<\zeta}\epsilon(r)\big| \in du\big|\zeta<\infty\big),
\end{aligned}\]
which could be explained by time-reversing process $X$.

Moreover, letting $p\to\infty$ in \eqref{n:pq:b<c} we have $n\big(\kappa_{b}^{+}=0\big)=0$ for every $b>0$.
Similarly, letting $p=q\to\infty$ in \eqref{n:pq:c<b} we have from \eqref{n:pq:bc}
\[\text{RHS}
=\frac{W^{(p)}(b-c)}{W^{(p)}(b)W^{(p)}(-c)}\E\big[e^{-p\tau_{c}^{-}}; \tau_{c}^{-}<\tau_{b}^{+}\big]
\le n\big(\zeta\ge \kappa_{b}^{+}\wedge \kappa_{c}^{-}\big)\cdot \E\big[e^{-p\tau_{c}^{-}}; \tau_{c}^{-}<\tau_{b}^{+}\big]\to0,
\]
that is, $n\big(\kappa_{c}^{-}=0<\zeta\wedge\kappa_{b}^{+}\big)=0$, which implies $n\big(\epsilon(0+)\neq0\big)=0$ as expected.
\end{rmk}
\begin{rmk}\label{rmk:excr}
Since an excursion is a piece of sample path for $X$ stopped at the first time of hitting $0$,
it holds that $\kappa_{0}^{-}=\zeta^{+}$ for lack of positive jumps; in addition, for $b>0>c$,
\[\begin{gathered}
\{\kappa_{b}^{+}<\zeta\}=\{\kappa_{b}^{+}<\infty\},\quad
\{\kappa_{c}^{-}<\zeta\}=\{\kappa_{c}^{-}<\infty\},\\
\{\zeta\le\kappa_{b}^{+}\}\cap\{\kappa_{c}^{-}<\zeta\}=\{\kappa_{c}^{-}<\zeta\wedge\kappa_{b}^{+}\},\\
\{\zeta\le\kappa_{b}^{+}\wedge\kappa_{c}^{-}\}
=\{\zeta<\kappa_{b}^{+}\wedge\kappa_{c}^{-}\}
\subset\{\zeta<\infty\}.
\end{gathered}\]
However, events $\{\kappa_{b}^{+}<\zeta=\infty\}, \{\zeta<\kappa_{b}^{+}=\infty\}$ and $\{\zeta=\kappa_{b}^{+}=\infty\}$ may not be empty events. Thus, we do not necessarily have $\{\zeta\le\kappa_{b}^{+}\}=\{\zeta<\kappa_{b}^{+}\}$, and so is the case for $\tau_{c}^{-}$.
\end{rmk}
\begin{rmk}
Although $\epsilon$ under $n$ is not a Markov process, it has the so called ``simple Markov property'' \cite[IV.4]{Bertoin96book}, that is, conditional on $\{\epsilon(t)=x, t<\zeta\}$, the shifted process $\{\epsilon(t+s), 0\le s<\zeta-t\}$ is independent of $\{\epsilon(s), 0\le s\le t\}$ and is distributed as $\{X(s), 0<s<\tau^{\{0\}}\}$ under $\P_{x}$, where $\tau^{\{0\}}:=\inf\{t>0, X(t)=0\}$ is the first hitting time of $0$ by $X$.
Moreover, it is often more convenient to work with the simple Markov property at the first passage times. For example, by \eqref{n:pq:b<c} we have for $0<x<b$,
\[
\frac{n\big(e^{-p \kappa_{b}^{+}}; \kappa_{b}^{+}<\zeta\big)}
{n\big(e^{-p \kappa_{x}^{+}}; \kappa_{x}^{+}<\zeta\big)}
=\E_{x}\Big[e^{-p\tau_{b}^{+}}; \tau_{b}^{+}<\tau_{0}^{-}\Big]
=\E_{x}\Big[e^{-p\tau_{b}^{+}}; \tau_{b}^{+}<\tau^{\{0\}}\Big].
\]
\end{rmk}

\section{Main results}\label{Main}
We first express Laplace transforms of the first passage times under $\P$ via excursion measure.
\begin{thm}\label{thm:1}
For $b>0>c, p>0$ and $r\in[0,\gamma]$, we have
\begin{align}
\label{eqn:n:r<bc}
\E\Big[e^{-p(\tau_{\gamma}-\gamma)}; \tau_{\gamma}<\tau_{b}^{+}\wedge\tau_{c}^{-}\Big]
=&\ \frac{n\big(e^{-p\zeta^{+}}; \zeta\le\kappa_{b}^{+}, \gamma<\zeta^{-}\wedge(\kappa_{c}^{-}-\zeta^{+})\big)}
{n\big(1-e^{-p\zeta}\mathbf{1}(\zeta<\kappa_{b}^{+}\wedge\kappa_{c}^{-},\zeta^{-}\le\gamma)\big)},\\
\label{eqn:n:b<cr}
\E\Big[e^{-p\tau_{b}^{+}}; \tau_{b}^{+}<\tau_{c}^{-}\wedge\tau_{\gamma}\Big]
=&\ \frac{n(e^{-p\kappa_{b}^{+}}; \kappa_{b}^{+}<\zeta)}
{n\big(1-e^{-p\zeta}\mathbf{1}(\zeta<\kappa_{b}^{+}\wedge\kappa_{c}^{-},\zeta^{-}\le\gamma)\big)},\\
\label{eqn:n:c<br}
\E\Big[e^{-p\tau_{c}^{-}}; \tau_{c}^{-}<\tau_{b}^{+}\wedge\tau_{\gamma}, l(\tau_{c}^{-})\le r\Big]
=&\ \frac{n\big(e^{-p\kappa_{c}^{-}};\kappa_{c}^{-}<\zeta\wedge\kappa_{b}^{+},\kappa_{c}^{-}-\zeta^{+}\le r\big)}
{n\big(1-e^{-p\zeta}\mathbf{1}(\zeta<\kappa_{b}^{+}\wedge\kappa_{c}^{-},\zeta^{-}\le\gamma)\big)}.
\end{align}
\end{thm}

In the proof, we enlarge the probability space so that $\{(\epsilon_{s},e_{p,s})\}_{s>0}$ is a Poisson point process on $\mathscr{E}\times(0,\infty)$ with characteristic measure
\begin{equation}\label{n:enlarge}
n_{e}\big(\epsilon\in B, e_{p}>t\big):=n(B)\cdot e^{-p t}\quad\forall B\in\mathscr{E}, t>0,
\end{equation}
for some $p>0$, in which given $\epsilon_{s}\neq \Delta $, $e_{p,s}$ is an independent exponential random variable with rate $p$. 
\begin{proof}
Being absent of positive jumps, for the excursion that the Parisian ruin happens, say $\epsilon_{t}$,
we must have
$$\zeta^{+}_{t}<\infty,\quad \zeta^{-}_{t}>\gamma,\quad \tau_{\gamma}-\gamma=L^{-1}(t-)+\zeta^{+}_{t}.$$
Therefore, we have from \eqref{d=0} that for such $t>0$
\begin{equation}
\label{iden:n:2}
\begin{aligned}
\tau_{\gamma}=\int_{0}^{\tau_{\gamma}}\mathbf{1}\big(s\in\bar{\mathcal{L}}\big)\,ds
+\int_{0}^{\tau_{\gamma}}\mathbf{1}\big(s\notin\bar{\mathcal{L}}\big)\,ds
=\int_{0}^{\tau_{\gamma}}\mathbf{1}\big(s\notin\bar{\mathcal{L}}\big)\,ds
=(\zeta^{+}_{t}+\gamma)+\sum_{s<t}\zeta_{s},
\end{aligned}
\end{equation}
which says that the Parisian run time is the sum of the time spent by $\epsilon_{t}$ up to the occurrence of Parisian ruin and the life times of those excursions before the index $t$,
where the time spent on $\bar{\mathcal{L}}$ is negligible.
Notice that we may have $\{\zeta_{t}^{+}<\infty,\zeta_{t}=\infty\}$ for some $t>0$, and we must have $\zeta^{-}_{s}\le\gamma$ and $\zeta_{s}<\infty$ for those excursions before the index $t$.
Therefore,
\[\begin{aligned}
 e^{-p(\tau_{\gamma}-\gamma)}\mathbf{1}\big(\tau_{\gamma}<&\ \tau_{b}^{+}\wedge\tau_{c}^{-}\big)
=\sum_{t}\Big(e^{-p\zeta^{+}_{t}}\mathbf{1}\big(\zeta^{+}_{t}<\infty, \gamma<\zeta_{t}^{-},
\zeta_{t}\le\kappa_{b,t}^{+}, \gamma<(\kappa_{c,t}^{-}-\zeta^{+}_{t})\big)\Big)\\
&\quad \times \exp\Big(-p\sum_{s<t}\zeta_{s}\Big)\mathbf{1}\big(\zeta_{s}<\infty,\zeta^{-}_{s}\le\gamma,
\zeta_{s}<\kappa_{b,s}^{+}\wedge\kappa_{c,s}^{-}, \forall s<t\big),
\end{aligned}\]
where the last term is left continuous in $t$. Applying the compensation formula gives
\begin{equation}\label{iden:n:1}
\begin{aligned}
\E\Big[e^{-p(\tau_{\gamma}-\gamma)}; \tau_{\gamma}<\tau_{b}^{+}\wedge\tau_{c}^{-}\Big]
=&\ \E\bigg[\int_{0}^{\infty}
n_{e}\Big(\zeta^{+}<e_{p},\zeta\le\kappa_{b}^{+}, \gamma<\zeta^{-}\wedge(\kappa_{c}^{-}-\zeta^{+})\Big)\\
&\quad \times \mathbf{1}\Big(\zeta_{s}<\kappa_{b,s}^{+}\wedge\kappa_{c,s}^{-}\wedge e_{p,s},\zeta_{s}^{-}\le\gamma,\forall s<t\Big)\,dt\bigg]\\
=&\ n\big(e^{-p\zeta^{+}}; \zeta\le\kappa_{b}^{+}, \gamma<\zeta^{-}\wedge(\kappa_{c}^{-}-\zeta^{+})\big)\cdot\int_{0}^{\infty}\P\big(\kappa_{A^{c}}>t\big)\,dt
\end{aligned}
\end{equation}
where the Poisson point process $\{(\epsilon_{t},e_{p,t})\}_{t>0}$ is specified in \eqref{n:enlarge},
\[
A:=\{\zeta<\kappa_{b}^{+}\wedge\kappa_{c}^{-}\}\cap\{\zeta^{-}\le\gamma\}\cap\{\zeta<e_{p}\},
\]
and $\kappa_{A^{c}}:=\inf\{s>0: (\epsilon_{s},e_{p,s})\notin A\}$ is the first entrance time of $A^{c}$, which is exponentially distributed with parameter
 \[
 n_{e}\big(1-\mathbf{1}(A)\big)
 =n\big(1-e^{-p\zeta}\mathbf{1}(\zeta<\kappa_{b}^{+}\wedge\kappa_{c}^{-},\zeta^{-}\le \gamma)\big).
 \]
Plugging the above into \eqref{iden:n:1} gives \eqref{eqn:n:r<bc}.

The same procedure can be applied to $\tau_{b}^{+}$ and $\tau_{c}^{-}$, and we only highlight the differences. On the sets $\{\tau_{b}^{+}<\infty\}$ and $\{\tau_{c}^{-}<\infty\}$, similar to \eqref{iden:n:2} we have from \eqref{d=0}
\[\begin{aligned}
e^{-p\tau_{b}^{+}}\mathbf{1}\big(\tau_{b}^{+}<&\ \tau_{c}^{-}\wedge\tau_{\gamma}\big)
= \sum_{t}\Big(e^{-p\kappa_{b,t}^{+}}\mathbf{1}\big(\kappa_{b,t}^{+}<\zeta_{t}\big)\Big)\\
&\quad \times\exp\Big(-p\sum_{s<t}\zeta_{s}\Big)\mathbf{1}\big(\zeta_{s}<\infty,\zeta^{-}_{s}\le\gamma,\zeta_{s}<\kappa_{b,s}^{+}\wedge\kappa_{c,s}^{-}, \forall s<t\big),\\
e^{-p\tau_{c}^{-}}\mathbf{1}\big(\tau_{c}^{-}<&\ \tau_{b}^{+}\wedge\tau_{\gamma}, l(\tau_{c}^{-})\le r\big)
= \sum_{t}\Big(e^{-p\kappa_{c,t}^{-}}\mathbf{1}\big(\zeta^{+}_{t}<\infty, \kappa_{c}^{-}-\zeta_{t}^{+}\le r,\zeta_{t}\le\kappa_{b,t}^{+}\big)\Big)\\
&\quad \times\exp\Big(-p\sum_{s<t}\zeta_{s}\Big)\mathbf{1}\big(\zeta_{s}<\infty,\zeta^{-}_{s}\le\gamma,\zeta_{s}<\kappa_{b,s}^{+}\wedge\kappa_{c,s}^{-}, \forall s<t\big),
\end{aligned}\]
respectively, where note that 
$$l(\tau_{c}^{-})=(\kappa_{c,t}^{-}-\zeta^{+}_{t}) \,\,\,\text{for}\,\,\, \tau_{c}^{-}<\infty$$ 
and 
$$\{\kappa_{c}^{-}<\infty\}\cap\{\zeta\le\kappa_{b}^{+}\}=\{\kappa_{c}^{-}<\zeta\wedge\kappa_{b}^{+}\}.$$
Then applying the same procedure as in \eqref{iden:n:1} and the compensation formula we can prove \eqref{eqn:n:b<cr} and \eqref{eqn:n:c<br}, respectively.
\end{proof}
\begin{rmk}
One can find that the formulas in Theorem \ref{thm:1} share the common denominator,
and the numerators describe what happen in the last excursion.
Given the weighed occupation times in Proposition \ref{prop:npq:bc}, one can also derive the weighed occupation times for the Parisian model under excursion measure $n$ by simply replacing $p\zeta$ with $\eta^{(p,q)}(\epsilon,\zeta)$. Moreover, the excursion method works for more general Markov models, the identities in the theorem also holds for reflected L\'evy process for example, in which the problem left is to characterise the excursion identities.
\end{rmk}

In the following, we present the Laplace transforms of the excursion identities in \eqref{eqn:n:r<bc}, \eqref{eqn:n:b<cr} and \eqref{eqn:n:c<br} by making use of Proposition \ref{prop:npq:bc}, where
the numerator in \eqref{eqn:n:b<cr} is already given by \eqref{n:pq:b<c}.
\begin{prop}\label{prop:cals}
For $b>0>c$ and $p,q>0$, we have
\begin{equation}
\int_{0}^{\infty}e^{-qt}n\Big(1-e^{-p\zeta}\mathbf{1}(\zeta<\kappa_{b}^{+}\wedge\kappa_{c}^{-}, \zeta^{-}\le t)\Big)\,dt=\frac{W^{(p,p+q)}(b,c)}{q\cdot W^{(p)}(b)W^{(p+q)}(-c)}\label{fun:1}.
\end{equation}
In addition,
\begin{align}
&\ \int_{0}^{\infty}e^{-qt}e^{-pt}n\Big(e^{-p\zeta^{+}}; \zeta\le \kappa_{b}^{+}, t<\zeta^{-}\wedge(\kappa_{c}^{-}-\zeta^{+})\Big)\,dt\non\\
=&\ \frac{1}{p+q}\Big(\frac{Z^{(p,p+q)}(b,c)-Z^{(p)}(b)}{W^{(p)}(b)}+\frac{W^{(p,p+q)}(b,c)(1-Z^{(p+q)}(-c))}{W^{(p)}(b)W^{(p+q)}(-c)}\Big),\label{fun:3}\\
&\ \int_{0}^{\infty}e^{-qt}n\Big(e^{-p\kappa_{c}^{-}}; \kappa_{c}^{-}<\zeta\wedge\kappa_{b}^{+}, \kappa_{c}^{-}-\zeta^{+}\le t\Big)\,dt\non\\
=&\ \frac{1}{q}\Big(\frac{Z^{(p+q)}(-c)W^{(p,p+q)}(b,c)}{W^{(p)}(b)W^{(p+q)}(-c)}
-\frac{Z^{(p,p+q)}(b,c)}{W^{(p)}(b)}\Big).\label{fun:4}
\end{align}
\end{prop}
\begin{proof}
Write $e_{q}$ for an exponential random variable independent of $X$.
We have
\[\begin{aligned}
&\ q\cdot \int_{0}^{\infty}e^{-qt}n\Big(1-e^{-p\zeta}\mathbf{1}\big(\zeta<\kappa_{b}^{+}\wedge\kappa_{c}^{-}, \zeta^{-}\le t\big)\Big)\,dt\\
=&\ \E\Big[n\Big(1-e^{-p\zeta}\mathbf{1}\big(\zeta<\kappa_{b}^{+}\wedge\kappa_{c}^{-}, \zeta^{-}\le e_{q}\big)\Big)\Big]
=n\Big(1-e^{-p\zeta-q\zeta^{-}}\mathbf{1}\big(\zeta<\kappa_{b}^{+}\wedge\kappa_{c}^{-}\big)\Big).
\end{aligned}\]
Then \eqref{fun:1} follows from \eqref{n:pq:bc}.

For the numerator in \eqref{eqn:n:r<bc}, we have
\[\begin{aligned}
&\ (p+q)\int_{0}^{\infty}e^{-qt} e^{-pt}n\big(e^{-p\zeta^{+}}; \zeta\le \kappa_{b}^{+}, t<\zeta^{-}\wedge(\kappa_{c}^{-}-\zeta^{+})\big)\,dt\non\\
=&\ \E\Big[n\Big(e^{-p\zeta^{+}}; \zeta\le \kappa_{b}^{+}, e_{p+q}<\zeta^{-}\wedge(\kappa_{c}^{-}-\zeta^{+})\Big)\Big]\non\\
=&\ n\Big(e^{-p\zeta^{+}}\big(1-e^{-(p+q)\zeta^{-}}\big); \zeta<\kappa_{b}^{+}\wedge\kappa_{c}^{-}\Big)
+n\Big(e^{-p\zeta^{+}}\big(1-e^{-(p+q)(\kappa_{c}^{-}-\zeta^{+})}\big); \kappa_{c}^{-}<\zeta\wedge\kappa_{b}^{+}\Big).
\end{aligned}\]
Then \eqref{fun:3} follows from \eqref{n:pq:bc} and \eqref{n:pq:c<b}.
For the numerator in \eqref{eqn:n:c<br}, we have
\[\begin{aligned}
&\ q\cdot \int_{0}^{\infty}e^{-qt}n\big(e^{-p\kappa_{c}^{-}}; \kappa_{c}^{-}<\zeta\wedge\kappa_{b}^{+}, \kappa_{c}^{-}-\zeta^{+}\le t\big)\,dt\\
=&\ \E\Big[n\big(e^{-p\kappa_{c}^{-}}; \kappa_{c}^{-}<\zeta\wedge\kappa_{b}^{+}, \kappa_{c}^{-}-\zeta^{+}\le e_{q}\big)\Big]
=n\big(e^{-p\kappa_{c}^{-}-q(\kappa_{c}^{-}-\zeta^{+})}; \kappa_{c}^{-}<\zeta\wedge\kappa_{b}^{+}\big),
\end{aligned}\]
and \eqref{fun:4} follows from \eqref{n:pq:c<b}. This completes the proof.
\end{proof}
\begin{rmk}
One may attempt to invert the Laplace transforms for the excursion quantities in Proposition \ref{prop:cals},
and find some explicit formulas for the passage problems in Theorem \ref{thm:1}
similarly to that in Proposition \ref{prop:ro18} of \cite{Loeffen2013,Loeffen2018}.
However, it is unlikely to find a general analytical expression at the presence the downward first passage time in the Parisian model. Taking the common denominator from \eqref{fun:1} in the case $p=0$ for example.

Letting $p\to0+$ in \eqref{fun:1}, applying \eqref{defn:wzpq} and \eqref{n:pq:b<c} we have
\begin{equation}\label{eqn:1}
\begin{aligned}
&\ \frac{W^{(0,q)}(b,c)}{q\cdot W^{(q)}(-c)}
= \frac{W(b-c)}{q\cdot W^{(q)}(-c)}+ \int_{c}^{0}W(b-z)\frac{W^{(q)}(z-c)}{W^{(q)}(-c)}\,dz\\
=&\ \frac{W(b-c)}{q}\cdot n\big(e^{-q\kappa_{-c}^{+}}; \kappa_{-c}^{+}<\zeta\big)+\int_{c}^{0}W(b-z)\E_{z}\big[e^{-q \tau_{0}^{+}}; \tau_{0}^{+}<\tau_{c}^{-}\big]\,dz\\
=&\ W(b)\cdot\int_{0}^{\infty}e^{-qt}n\big(1-\mathbf{1}(\zeta<\kappa_{b}^{+}\wedge\kappa_{c}^{-}, \zeta^{-}\le t)\big)dt.
\end{aligned}
\end{equation}
Thus, inverting the Laplace transform above and taking $q^{-1}$ as the Laplace transform of the identity function gives that for almost every $t>0$,
\[\begin{aligned}
&\ W(b)\cdot n\big(1-\mathbf{1}(\zeta<\kappa_{b}^{+}\wedge\kappa_{c}^{-}, \zeta^{-}\le t)\big)\\
=&\ W(b-c)\cdot n\big(\kappa_{-c}^{+}\le\zeta\wedge t\big)+\frac{\partial}{\partial t}\Big(\int_{c}^{0}W(b-z)\P_{z}\big(\tau_{0}^{+}<\tau_{c}^{-}\wedge t\big)\,dz\Big).
\end{aligned}\]
However, letting $c\to-\infty$ in \eqref{eqn:1}, we have
\[
\frac{W^{(0,q)}(b,c)}{q\cdot W^{(q)}(-c)}\to \int_{0}^{\infty}W(b+z)e^{-\phi(q)z}dz=\int_{0}^{\infty}e^{-q t}\Lambda(b,t)dt
\]
applying \eqref{ken:w}. Then by the continuity of the functions and inverting the Laplace transforms in \eqref{eqn:1} we have
\[
W(b)\cdot n\big(1-\mathbf{1}(\zeta<\kappa_{b}^{+}, \zeta^{-}\le t)\big)=\Lambda(b,t).
\]
More limits for the case $c=-\infty$ will be examined in section \ref{recover}, where results in \cite{Loeffen2013} and \cite{Loeffen2018} are recovered.
\end{rmk}

\begin{rmk}
There are some other interesting identities can be derived from Propositions \ref{prop:cals} and \ref{prop:pq:ab}.
For example, we have from \eqref{fun:1} and \eqref{lapocc:b<c} that
\[
\int_{0}^{\infty}e^{-qt}n\Big(1-e^{-p\zeta}\mathbf{1}\big(\zeta<\kappa_{b}^{+}\wedge\kappa_{c}^{-}, \zeta^{-}\le (t-s)\big)\Big)dt
\times 
\E\Big[e^{-p\tau_{b}^{+}-q\mom(\tau_{b}^{+})}; \tau_{b}^{+}\le \tau_{c}^{-}\Big]
=\frac{1}{qW^{(p)}(b)}
\]
since $p\mop(t)+(p+q)\mom(t)=pt+q\mom(t)$ and $\wpq(0,c)=W^{(p+q)}(-c)$,
which implies
\[
\int_{0}^{t}n\Big(1-e^{-p\zeta}\mathbf{1}\big(\zeta<\kappa_{b}^{+}\wedge\kappa_{c}^{-}, \zeta^{-}\le (t-s)\big)\Big) \E\Big[e^{-p\tau_{b}^{+}}; \tau_{b}^{+}<\tau_{c}^{-}, \mom(\tau_{b}^{+})\in ds\Big]
=\frac{t}{W^{(p)}(b)},
\]
by inverting the Laplace transform in $q$.
Similar identities can be found from other quantities.
\end{rmk}



To derive the Laplace transforms for surplus process with general initial value, we apply the Markov property in Theorem \ref{thm:1}.
Recall from Remark \ref{rmk_markov} that
$X(\cdot\wedge\tau_{\gamma})$ is not a Markov process on $\RR$, but $(X,l)$ is.
Thus, in the following we would like to present the result for $(X(0),l(0))=(x,t)$.
Recalling that $l(0)=0$ if $X(0)\ge0$ and if $X(0)=x<0$ and $l(0)=t$, we have $l(0+)=t$ and in addition,
\[
\tau_{\gamma}:=\inf\{t>0: l(t)>\gamma\}.
\]



For $p>0>c, x>c$ and $t>0$, define functions
\begin{align}
H^{(p)}(t,x,c):=&\
\left\{\begin{array}{l@{\quad\text{if}\quad}l}
W^{(p)}(x)\cdot n\big(1-e^{-p\zeta}\mathbf{1}(\zeta<\kappa_{x}^{+}\wedge\kappa_{c}^{-}, \zeta^{-}\le t)\big) & x>0,\\
1 & x=0,\\
\E_{x}\Big[e^{-p\tau_{0}^{+}}; \tau_{0}^{+}<\tau_{c}^{-}\wedge t\Big] & x\in(c,0);
\end{array}\right.\label{defn:h}\\
J^{(p)}(t,x,c):=&\
\left\{\begin{array}{l@{\quad\text{if}\quad}l}
W^{(p)}(x)\cdot e^{-pt} n\big(e^{-p\zeta^{+}}; \zeta\le\kappa_{x}^{+}, t<\zeta^{-}\wedge(\kappa_{c}^{-}-\zeta^{+})\big)& x>0,\\
0 & x=0,\\
-e^{-pt}\P_{x}\big(t<\tau_{0}^{+}\wedge\tau_{c}^{-}\big) & x\in(c,0);
\end{array}\right.\label{defn:j}\\
K^{(p)}(t,x,c):=&\
\left\{\begin{array}{l@{\quad\text{if}\quad}l}
W^{(p)}(x)\cdot n\big(e^{-p\zeta_{c}^{-}}; \kappa_{c}^{-}\le\zeta\wedge\kappa_{x}^{+}\wedge(\zeta^{+}+t)\Big) & x>0,\\
0& x=0,\\
-\E_{x}\Big[e^{-p\tau_{c}^{-}}; \tau_{c}^{-}<\tau_{0}^{+}\wedge t\Big]& x\in(c,0).
\end{array}\right.\label{defn:k}
\end{align}

\begin{prop}\label{prop:funs}
For every $p>0, c<0$ and $x>c$, the functions $H^{(p)},J^{(p)},K^{(p)}$ are right continuous in $t$ such that for any $ q>0$,
\begin{gather}
\int_{0}^{\infty}e^{-qt}H^{(p)}(t,x,c)\,dt
=\frac{W^{(p,p+q)}(x,c)}{q\cdot W^{(p+q)}(-c)},\label{defn:h2}\\
\int_{0}^{\infty}e^{-qt}J^{(p)}(t,x,c)\,dt
=\frac{Z^{(p,p+q)}(x,c)-Z^{(p)}(x)}{p+q}+\frac{W^{(p,p+q)}(x,c)(1-Z^{(p+q)}(-c))}{(p+q)\cdot W^{(p+q)}(-c)},\label{defn:j2}\\
\int_{0}^{\infty}e^{-qt}K^{(p)}(t,x,c)\,dt
=\frac{1}{q}\Big(\frac{W^{(p,p+q)}(x,c)}{W^{(p+q)}(-c)}Z^{(p+q)}(-c)-Z^{(p,p+q)}(x,c)\Big).\label{defn:k2}
\end{gather}
\end{prop}

\begin{proof}

For the case $x>0$, the Laplace transforms are already obtained in Proposition \ref{prop:cals}.
The case $x=0$ can be checked directly.

For $x<0$, 
it follows from
\[Z^{(p,p+q)}(x,c)=Z^{(p+q)}(x-c), W^{(p,p+q)}(x,c)=W^{(p+q)}(x-c)\,\, \text{and}\,\, Z^{(p)}(x)=1\]
 that
\[\begin{aligned}
&\ q \int_{0}^{\infty}e^{-qt}H^{(p)}(t,x,c)\,dt
=\E_{x}\Big[e^{-(p+q)\tau_{0}^{+}}; \tau_{0}^{+}<\tau_{c}^{-}\Big]
=\frac{W^{(p+q)}(x-c)}{W^{(p+q)}(-c)}=\frac{W^{(p,p+q)}(x,c)}{W^{(p+q)}(-c)},\\
&\ q\int_{0}^{\infty}e^{-qt}K^{(p)}(t,x,c)\,dt=-\E_{x}\Big[e^{-(p+q)\tau_{c}^{-}}; \tau_{c}^{-}<\tau_{0}^{+}\Big]\\
=&\ \frac{W^{(p+q)}(x-c)}{W^{(p+q)}(-c)}Z^{(p+q)}(-c)-Z^{(p+q)}(x-c)
=\frac{W^{(p,p+q)}(x,c)}{W^{(p+q)}(-c)}Z^{(p+q)}(-c)-Z^{(p,p+q)}(x,c),\\
&\ (p+q)\int_{0}^{\infty}e^{-qt}J^{(p)}(t,x,c)\,dt
=-\P_{x}\big(e_{p+q}<\tau_{0}^{+}\wedge\tau_{c}^{-}\big)\,dt
=\E_{x}\Big[e^{-(p+q)(\tau_{0}^{+}\wedge\tau_{c}^{-})}\Big]-1\\
=&\ \frac{W^{(p,p+q)}(x,c)}{W^{(p+q)}(-c)}+Z^{(p,p+q)}(x,c)-\frac{W^{(p,p+q)}(x,c)}{W^{(p+q)}(-c)}Z^{(p+q)}(-c)-1\\
=&\ \Big(Z^{(p,p+q)}(x,c)-Z^{(p)}(x)\Big)+\frac{W^{(p,p+q)}(x,c)(1-Z^{(p,p+q)}(-c))}{W^{(p+q)}(-c)},
\end{aligned}\]
which lead to the Laplace transforms for $x<0$ and completes the proof.
\end{proof}

Using the functions $H^{(p)},J^{(p)},K^{(p)}$ we obtain expressions for the following Laplace transforms.
\begin{thm}\label{thm:lapx}
Given $b>0>c$ and $p>0$, we have for any $x\in(c,b)$ and any $r,t\in[0,\gamma)$
\begin{gather}
\label{lapx:b<cr}
\E_{x,t}\Big[e^{-p\tau_{b}^{+}}; \tau_{b}^{+}<\tau_{c}^{-}\wedge\tau_{\gamma}\Big]=\frac{H^{(p)}(\gamma-t,x,c)}{H^{(p)}(\gamma,b,c)},\\
\label{lapx:r<bc}
\E_{x,t}\Big[e^{-p\tau_{\gamma}}; \tau_{\gamma}<\tau_{b}^{+}\wedge\tau_{c}^{-}\Big]
=\frac{H^{(p)}(\gamma-t,x,c)}{H^{(p)}(\gamma,b,c)}J^{(p)}(\gamma,b,c)-J^{(p)}(\gamma-t,x,c),\\
\label{lapx:c<br}
\E_{x,t}\Big[e^{-p\tau_{c}^{-}}; l(\tau_{c}^{-})\le r, \tau_{c}^{-}<\tau_{b}^{+}\wedge\tau_{\gamma}\Big]
=\frac{H^{(p)}(\gamma-t,x,c)}{H^{(p)}(\gamma,b,c)}K^{(p)}(r,b,c)-K^{(p)}(r-t,x,c),
\end{gather}
where we take $K^{(p)}(s,x,c)=0$ for $s\le0$.
\end{thm}

\begin{proof}
We first consider the case of $x>0$ and $t=0$ for which 
we apply the Markov property of $X$ at $\tau_{x}^{+}$ to Theorem \ref{thm:1}.
Under $\P$ and for shift operator $\theta_t$, we have from $\tau_{b}^{+}=\tau_{x}^{+}+\tau_{b}^{+}\circ\theta_{\tau_{x}^{+}}$ that
\[
\E\Big[e^{-p\tau_{b}^{+}}; \tau_{b}^{+}<\tau_{c}^{-}\wedge\tau_{\gamma}\Big]
=\E\Big[e^{-p\tau_{x}^{+}}; \tau_{x}^{+}<\tau_{c}^{-}\wedge\tau_{\gamma}\Big]
\cdot\E_{x}\Big[e^{-p\tau_{b}^{+}}; \tau_{b}^{+}<\tau_{c}^{-}\wedge\tau_{\gamma}\Big],
\]
applying \eqref{eqn:n:b<cr} to the above equation and making substitutions with \eqref{n:pq:b<c} and \eqref{defn:h}
gives \eqref{lapx:b<cr}.
Similarly, we have
\[\begin{aligned}
\E\Big[e^{-p\tau_{\gamma}};&\ \tau_{\gamma}<\tau_{b}^{+}\wedge\tau_{c}^{-}\Big]
=\E\Big[e^{-p\tau_{\gamma}}; \tau_{\gamma}<\tau_{x}^{+}\wedge\tau_{c}^{-}\Big]\\
&\quad + \E\Big[e^{-p\tau_{x}^{+}}; \tau_{x}^{+}<\tau_{c}^{-}\wedge\tau_{\gamma}\Big]\cdot
\E_{x}\Big[e^{-p\tau_{\gamma}}; \tau_{\gamma}<\tau_{b}^{+}\wedge\tau_{c}^{-}\Big],\\
\E\Big[e^{-p\tau_{c}^{-}};&\ l(\tau_{c}^{-})\le r, \tau_{c}^{-}<\tau_{b}^{+}\wedge\tau_{\gamma}\Big]
=\E\Big[e^{-p\tau_{c}^{-}}; l(\tau_{c}^{-})\le r, \tau_{c}^{-}<\tau_{x}^{+}\wedge\tau_{\gamma}\Big]\\
&\quad +\E\Big[e^{-p\tau_{x}^{+}}; \tau_{x}^{+}<\tau_{c}^{-}\wedge\tau_{\gamma}\Big]\cdot
\E_{x}\Big[e^{-p\tau_{c}^{-}}; l(\tau_{c}^{-})\le r, \tau_{c}^{-}<\tau_{b}^{+}\wedge\tau_{\gamma}\Big].
\end{aligned}\]
Then applying \eqref{eqn:n:r<bc}, \eqref{eqn:n:b<cr}, \eqref{eqn:n:c<br} and making substitutions using \eqref{defn:j} and \eqref{defn:k} gives \eqref{lapx:r<bc} and \eqref{lapx:c<br}.

For the case $(X,l)=(x,t)$ with $x\in(c,0]$ and $t\in[0,\gamma)$,
we have $\tau_{b}^{+}>\tau_{0}^{+}$ and only need to compare $\tau_{c}^{-}, \tau_{0}^{+}$ and $\gamma-t$, then everything is renewed at $\tau_{0}^{+}$ by the Markov property, that is,
\[
\E_{x,t}\Big[e^{-p\tau_{b}^{+}}; \tau_{b}^{+}<\tau_{c}^{-}\wedge\tau_{\gamma}\Big]
=\E_{x}\Big[e^{-p\tau_{0}^{+}}; \tau_{0}^{+}<\tau_{c}^{-}\wedge(\gamma-t)\Big]
\cdot\E\Big[e^{-p\tau_{b}^{+}}; \tau_{b}^{+}<\tau_{c}^{-}\wedge\tau_{\gamma}\Big].
\]
Substituting the last expressions of in \eqref{defn:h} for $H^{(p)}$ gives \eqref{lapx:b<cr}.
Similarly, we have
\[\begin{aligned}
\E_{x,t}\Big[e^{-p\tau_{\gamma}};&\ \tau_{\gamma}<\tau_{b}^{+}\wedge\tau_{c}^{-}\Big]
=e^{-p(\gamma-t)}\P_{x}\big((\gamma-t)<\tau_{0}^{+}\wedge\tau_{c}^{-}\big)\\
&\quad + \E_{x}\Big[e^{-p\tau_{0}^{+}}; \tau_{0}^{+}<\tau_{c}^{-}\wedge(\gamma-t)\Big]
\cdot \E\Big[e^{-p\tau_{\gamma}}; \tau_{\gamma}<\tau_{b}^{+}\wedge\tau_{c}^{-}\Big]\\
&\quad=-J^{(p)}(\gamma-t,x,c)+H^{(p)}(\gamma-t,x,c)\cdot \E\Big[e^{-p\tau_{\gamma}}; \tau_{\gamma}<\tau_{b}^{+}\wedge\tau_{c}^{-}\Big].
\end{aligned}\]
Making substitutions using \eqref{defn:h} and \eqref{defn:j} we prove \eqref{lapx:r<bc}.
Finally, since
\[\begin{aligned}
\E_{x,t}\Big[e^{-p\tau_{c}^{-}};&\ l(\tau_{c}^{-})\le r, \tau_{c}^{-}<\tau_{b}^{+}\wedge\tau_{\gamma}\Big]
= \E_{x}\Big[e^{-p\tau_{c}^{-}}; \tau_{c}^{-}<\tau_{b}^{+}\wedge(r-t)\Big]\\
&\quad +\E_{x}\Big[e^{-p\tau_{0}^{+}}; \tau_{0}^{+}<\tau_{c}^{-}\wedge(\gamma-t)\Big]
\cdot \E\Big[e^{-p\tau_{c}^{-}}; l(\tau_{c}^{-})\le r, \tau_{c}^{-}<\tau_{b}^{+}\wedge\tau_{\gamma}\Big]\\
&\quad = -K^{(p)}(r-t,x,c)+H^{(p)}(\gamma-t,x,c)
\cdot \E\Big[e^{-p\tau_{c}^{-}}; l(\tau_{c}^{-})\le r, \tau_{c}^{-}<\tau_{b}^{+}\wedge\tau_{\gamma}\Big],
\end{aligned}\]
making substitutions using \eqref{defn:h} and \eqref{defn:k} we prove \eqref{lapx:c<br}. This completes the proof.
\end{proof}

We also obtain the following resolvent measure under $\P$, and leave that under $\P_{(x,t)}$ to interested readers where heavier notations are expected.
\begin{thm}\label{thm:resl}
For $b>0>c, p>0$ and $r\in(0,\gamma]$, we have
\begin{equation}\label{resl:np}
\begin{aligned}
&\ \int_{0}^{\infty}e^{-pt}\E\Big[f\big(X(t)\big); l(t)\le r, t<\tau_{b}^{+}\wedge\tau_{c}^{-}\wedge\tau_{\gamma}\Big]\,dt\\
=&\ n\Big(\int_{0}^{\zeta\wedge\kappa_{b}^{+}\wedge\kappa_{c}^{-}\wedge(\zeta^{+}+r)}e^{-pt}f\big(\epsilon(t)\big)\,dt\Big)\Big/n\Big(1-e^{-p\zeta}\mathbf{1}(\zeta<\kappa_{b}^{+}\wedge\kappa_{c}^{-},\zeta^{-}\le \gamma)\Big).
\end{aligned}
\end{equation}
for ever bounded and measurable function $f$, and for every $q>0$ we have
\begin{equation}\label{fun:2}
\begin{aligned}
&\ \int_{0}^{\infty}e^{-qt}n\Big(\int_{0}^{\zeta\wedge\kappa_{b}^{+}\wedge\kappa_{c}^{-}\wedge(\zeta^{+}+t)}e^{-ps}f\big(\epsilon(s)\big)\,ds\Big)\,dt\\
=&\ \frac{1}{q}\int_{c}^{b}f(y)\Big(\frac{W^{(p,p+q)}(b,y)}{W^{(p)}(b)}-\frac{W^{(p,p+q)}(b,c)W^{(p+q)}(-y)}{W^{(p)}(b)W^{(p+q)}(-c)}\Big)\,dy.
\end{aligned}
\end{equation}
\end{thm}
\begin{proof}
Let $e_{p}$ be an exponential variable independent to $X$. Then $\P\big(X(e_{p})=0\big)=1$.
For $t=L(e_{p})$, $L^{-1}(t-)<e_{p}<L^{-1}(t), X(e_{p})=\epsilon_{t}(e_{p}-L^{-1}(t-))$ and $l(e_{p})=(e_{p}-L^{-1}(t-)-\zeta_{t}^{+})^{+}$. Then
\[\begin{aligned}
&\ f\big(X(e_{p})\big)\mathbf{1}\big(l(e_{p})\le r, e_{p}<\tau_{b}^{+}\wedge\tau_{c}^{-}\wedge\tau_{\gamma}\big)\\
=&\ \sum_{t}\mathbf{1}\big(L^{-1}(t-)<e_{p}<L^{-1}(t)\big)\cdot\mathbf{1}\big(e_{p}-L^{-1}(t-)<\zeta_{t}\wedge\kappa_{b,t}^{+}\wedge\kappa_{c,t}^{-}\wedge(\zeta^{+}_{t}+r)\big)\\
&\quad \times \mathbf{1}\big(\zeta_{s}<\infty,\zeta_{s}<\kappa_{b,s}^{+}\wedge\kappa_{c,s}^{-},\zeta_{s}^{-}\le\gamma, \forall s<t\big)\cdot f\big(\epsilon_{t}(e_{p}-L^{-1}(t-))\big).
\end{aligned}\]

Noticing that the sum above is over a countable set of $t$, making use of the memoryless property of $e_{p}$ and the fact $L^{-1}(t-)=\sum_{s<t}\zeta_{s}$ used in \eqref{iden:n:2}, we have
\[\begin{aligned}
&\ \E\Big[f\big(X(e_{p})\big); l(e_{p})\le r, e_{p}<\tau_{b}^{+}\wedge\tau_{c}^{-}\wedge\tau_{\gamma}\Big]\\
=&\ \E\bigg[\sum_{t}f\big(\epsilon_{t}(e_{p})\big)\mathbf{1}\big(e_{p}<\zeta_{t}\wedge\kappa_{b,t}^{+}\wedge\kappa_{c,t}^{-}\wedge(\zeta^{+}_{t}+r)\big)\\
&\quad \times \exp\Big(-p\sum_{s<t}\zeta_{s}\Big)\mathbf{1}\big(\zeta_{s}<\infty,\zeta_{s}<\kappa_{b,s}^{+}\wedge\kappa_{c,s}^{-},\zeta_{s}^{-}\le\gamma, \forall s<t\big)\bigg]\\
=&\ \E\bigg[\sum_{t}f\big(\epsilon_{t}(e_{p,t})\big)\mathbf{1}\big(e_{p,t}<\zeta_{t}\wedge\kappa_{b,t}^{+}\wedge\kappa_{c,t}^{-}\wedge(\zeta^{+}+r)\big)\times \mathbf{1}\big(\kappa_{A^{c}}>t\big)\bigg],
\end{aligned}\]
where in the last identity $\{(\epsilon_{t},e_{p,t})\}_{t>0}$ is the enriched Poisson point process, and $\kappa_{{A}^{c}}$ is the first entrance time used in \eqref{iden:n:1}.
Then, the compensation formula can be applied and gives
\[\begin{aligned}
&\ p \int_{0}^{\infty}e^{-pt}\E\Big[f(X(t)); l(t)\le r, t<\tau_{b}^{+}\wedge\tau_{c}^{-}\wedge\tau_{\gamma}\Big]\,dt\\
=&\ n_{e}\Big(f\big(\epsilon(e_{p})\big); e_{p}<\zeta\wedge\kappa_{b}^{+}\wedge\kappa_{c}^{-}\wedge(\zeta^{+}+r)\Big)\Big/ n_{e}\big(1-\mathbf{1}(A)\big)
\end{aligned}\]
which proves \eqref{resl:np}.
For the numerators in \eqref{resl:np}, we have
\[\begin{aligned}
&\ q\cdot \int_{0}^{\infty}e^{-qr}n\Big(\int_{0}^{\zeta\wedge\kappa_{b}^{+}\wedge\kappa_{c}^{-}\wedge(\zeta^{+}+r)}e^{-pt}f\big(\epsilon(t)\big)\,dt\Big)\,dr\\
=&\ \E\Big[n\Big(\int_{0}^{\zeta\wedge\kappa_{b}^{+}\wedge\kappa_{c}^{-}\wedge(\zeta^{+}+e_{q})}e^{-pt}f\big(\epsilon(t)\big)\,dt\Big)\Big]
=n\Big(\int_{0}^{\zeta\wedge\kappa_{b}^{+}\wedge\kappa_{c}^{-}}e^{-pt}e^{-q(t-\zeta^{+})^{+}}f\big(\epsilon(t)\big)\,dt\Big)
\end{aligned}\]
applying \eqref{n:pq:resl} proves \eqref{fun:2}.
\end{proof}

\section{Recovery of the previous Parisian ruin results}\label{recover}
In this section we recover the results in \cite{Loeffen2013} and \cite{Loeffen2018}.

For $y\ge0$, we have in \eqref{fun:2}
\[
\frac{W^{(p,p+q)}(b,y)}{W^{(p)}(y)}-\frac{W^{(p,p+q)}(b,c)W^{(p+q)}(-y)}{W^{(p)}(b)W^{(p+q)}(-c)}=\frac{W^{(p)}(b-y)}{W^{(p)}(b)}.
\]
Therefore, inverting the Laplace transform and plugging into \eqref{resl:np} we always have for $y\ge0$
\[
\int_{0}^{\infty}e^{-ps}\P\Big(X(s)\in dy, l(s)\le r, s<\tau_{b}^{+}\wedge\tau_{c}^{-}\wedge\tau_{\gamma}\Big)\,ds
=\frac{W^{(p)}(b-y)}{H^{(p)}(\gamma,b,c)},
\]
noticing that $l(s)=0$ for $X(s)=y\ge0$. 
Then applying the Markov property, similar to the proof of Theorem \ref{thm:lapx}, we find for $y\ge0$ and $x\in(c,b)$ and $t,r\in[0,\gamma)$
\[\begin{aligned}
&\ \int_{0}^{\infty}e^{-ps}\P_{(x,t)}\Big(X(s)\in dy, l(s)\le r, s<\tau_{b}^{+}\wedge\tau_{c}^{-}\wedge\tau_{\gamma}\Big)\,ds\\
=&\ \Big(\frac{H^{(p)}(\gamma-t,x,c)}{H^{(p)}(\gamma,b,c)}W^{(p)}(b-y)-W^{(p)}(x-y)\Big)\,dy,
\end{aligned}\]
which recovers \cite[Prop.3.4]{Loeffen2018}.

We now further recover the results of \cite{Loeffen2013,Loeffen2018} summarized in Proposition \ref{prop:ro18} by letting $c\to-\infty$.
Recalling $\Lambda^{(p)}$ defined in \eqref{ken:w}, we first claim that for $x\in\mathbb{R}$
\begin{gather}
\lim_{c\to-\infty}H^{(p)}(t,x,c)=e^{-pt}\Lambda^{(p)}(x,t)\label{limit:H},\\
\lim_{c\to-\infty}J^{(p)}(t,x,c)=e^{-pt}\Big(\Lambda^{(p)}(x,t)- p \int_{0}^{t}\Lambda^{(p)}(x,s)\,ds- Z^{(p)}(x)\Big).\label{limit:J}
\end{gather}
It then follows from Theorem \ref{thm:lapx} that
\[\begin{gathered}
\E_{x,t}\Big[e^{-p\tau_{b}^{+}}; \tau_{b}^{+}<\tau_{\gamma}\Big]=e^{-pt}\frac{\Lambda^{(p)}(x,\gamma-t)}{\Lambda^{(p)}(b,\gamma)},\\
\begin{aligned}
\E_{x,t}\Big[e^{-p(\tau_{\gamma}-\gamma)}; \tau_{\gamma}<\tau_{b}^{+}\Big]
=&\ e^{-pt}\Big(Z^{(p)}(x)+ p\int_{0}^{\gamma-t}\Lambda^{(p)}(x,s)\,ds\Big)\\
&\quad -e^{-pt}\frac{\Lambda^{(p)}(x,\gamma-t)}{\Lambda^{(p)}(b,\gamma)}\Big(Z^{(p)}(b)+ p\int_{0}^{\gamma}\Lambda^{(p)}(b,s)\,ds\Big),
\end{aligned}
\end{gathered}\]
which coincides with \eqref{prop:ro18:b<r} and \eqref{prop:ro18:r<b} by taking $t=0$.

To show the limits \eqref{limit:H} and \eqref{limit:J}, by the definitions of $H^{(p)}$ and $J^{(p)}$ in \eqref{defn:h} and \eqref{defn:j},
we can invert the limits of the Laplace transforms in Proposition \ref{prop:funs}.
From the identities \eqref{defn:wzpq} and \eqref{wz-limit}, we have for $x>c$ in \eqref{defn:h2}
\[\begin{aligned}
\frac{W^{(p,p+q)}(x,c)}{W^{(p+q)}(-c)}=&\ \frac{W^{(p+q)}(x-c)}{W^{(p+q)}(-c)}- q\int_{0}^{x}W^{(p)}(x-z)\frac{W^{(p+q)}(z-c)}{W^{(p+q)}(-c)}\,dz\\
\to&\ e^{\phi(p+q)x}-q\int_{0}^{x}e^{\phi(p+q)z}W^{(p)}(x-z)\,dz.
\end{aligned}\]
as $c\to-\infty$, where the limit above also holds for the case $x<0$
in which each integral equals to $0$ since $W(x-z)=0$ for $z\in(x,0)$ by definition.
Further noticing that
$$\D q\int_{0}^{\infty}e^{-\phi(p+q)z}W^{(p)}(z)\,dz=1,$$
 we have by change of variable
\begin{equation}
\label{lim1}
\frac{W^{(p,p+q)}(x,c)}{W^{(p+q)}(-c)}
\to q \int_{0}^{\infty}e^{-\phi(p+q)z}W^{(p)}(z+x)\,dz
=q \int_{0}^{\infty}e^{-(p+q)t}\Lambda^{(p)}(x,t)\,dt
\end{equation}
as $c\to -\infty$, where \eqref{ken:w} is applied for the last equality. Therefore,
 \eqref{limit:H} holds first for almost all $t>0$ and then for all $t>0$ by continuity.

For the Laplace transform of $J^{(p)}$ in \eqref{defn:j2}, by \eqref{lim1} and \eqref{lapocc:c<b} we have as $c\to-\infty$, for $x>0$
\[\begin{aligned}
&\ Z^{(p,p+q)}(x,c)-\frac{\wpq(x,c)}{W^{(p+q)}(-c)}Z^{(p+q)}(-c)\\
=&\ \frac{W^{(p,p+q)}(x,c)}{-W^{(p+q)}(-c)}\Big(Z^{(p+q)}(-c)-\frac{W^{(p+q)}(-c)}{W^{(p,p+q)}(x,c)}Z^{(p,p+q)}(x,c)\Big)\\
=&\ \frac{W^{(p,p+q)}(x,c)}{-W^{(p+q)}(-c)}
\E\Big[\exp\big(-\mathcal{O}^{(p,p+q)}(\tau_{c}^{-})\big); \tau_{c}^{-}<\tau_{x}^{+}\Big]
\to 0;
\end{aligned}\]
 while for $x<0$,
\[\begin{aligned}
&\ Z^{(p,p+q)}(x,c)-\frac{\wpq(x,c)}{W^{(p+q)}(-c)}Z^{(p+q)}(-c)\\
=&\ Z^{(p+q)}(x-c)-\frac{W^{(p+q)}(x-c)}{W^{(p+q)}(-c)}Z^{(p+q)}(-c)\\
=&\ \E_{x}\Big[e^{-(p+q)\tau_{c}^{-}}; \tau_{c}^{-}<\tau_{0}^{+}\Big]\to 0.
\end{aligned}\]
Thus, in both cases we have from \eqref{lim1} that 
\[
\text{RHS of \eqref{defn:j2}}\to\frac{1}{p+q}
\Big(q\int_{0}^{\infty}e^{-(p+q)t}\Lambda^{(p)}(x,t)\,dt-Z^{(p)}(x)\Big).
\]
Inverting the Laplace transform gives \eqref{limit:J}.

\appendix
\section*{Appendix}
\setcounter{equation}{0}
\setcounter{thm}{0}
\renewcommand{\theequation}{A.\arabic{equation}}
\renewcommand{\thethm}{A.\arabic{thm}}
Similar to the proof of Theorem \ref{thm:1}, we rewrite the variables in Proposition \ref{prop:pq:ab} first and then the compensation formula is applied, where we only highlight the difference.
\begin{proof}[Proof of Proposition \ref{prop:npq:bc}]
Recall the weighted occupation time $\eta^{(p,q)}$ defined for excursion.
On the set $\{L^{-1}(t)<\infty\}$, we have $\zeta_{s}<\infty$ for all $s\le t$, in terms of excursion
\begin{equation}\label{eqn:A1}
\begin{aligned}
&\ \exp\Big(-\mo\big(L^{-1}(t)\big)\Big)\mathbf{1}\big(L^{-1}(t)<\tau_{b}^{+}\wedge\tau_{c}^{-}\big)\\
=&\ \exp\Big(-\sum_{s\le t}\eta^{(p,q)}(\epsilon_{s},\zeta_{s})\Big)\mathbf{1}\big(\zeta_{s}<\kappa_{b,s}^{+}\wedge\kappa_{c,s}^{-},\forall s\le t\big)\\
=&\ \lim_{r\to0+}\exp\Big(-\sum_{s\in I_{r}(t)}\eta^{(p,q)}(\epsilon_{s},\zeta_{s})\Big)\mathbf{1}\big(\zeta_{s}<\kappa_{b,s}^{+}\wedge\kappa_{c,s}^{-}, s\in I_{r}(t)\big)
\end{aligned}
\end{equation}
where $I_{r}(t):=\{s\le t:\zeta_{s}>r\}$ which is Poisson distributed with parameter $\lambda=n(\zeta>r)\cdot t$, and given $I_{r}(t)$, $\{\epsilon_{s}, s\in I_{r}(t)\}$ are i.i.d. with law $n_{r}:=n(\cdot|\zeta>r)$. It then follows that
\[\begin{aligned}
&\ \E\bigg[\exp\Big(-\sum_{s\in I_{r}(t)}\eta^{(p,q)}(\epsilon_{s},\zeta_{s})\Big)\mathbf{1}\big(\zeta_{s}<\kappa_{b,s}^{+}\wedge\kappa_{c,s}^{-}, s\in I_{r}(t)\big)\bigg]\\
=&\ \sum_{k\ge0}\frac{\lambda^{k}}{k!}e^{-\lambda}\Big(n_{r}\Big[e^{-\eta^{(p,q)}(\epsilon,\zeta)}; \zeta<\kappa_{b}^{+}\wedge\kappa_{c}^{-}\Big]\Big)^{k}\\
=&\ \exp\Big(-\lambda\Big(1- \int_{\mathscr{E}}e^{-\eta^{(p,q)}(\epsilon,\zeta)}\mathbf{1}(\zeta<\kappa_{b}^{+}\wedge\kappa_{c}^{-})dn_{r}(\epsilon)\Big)\Big)\\
=&\ \exp\Big(-t\int_{\mathscr{E}}\Big(1-e^{-\eta^{(p,q)}(\epsilon,\zeta)}\mathbf{1}(\zeta<\kappa_{b}^{+}\wedge\kappa_{c}^{-})\Big)\mathbf{1}(\zeta>r)dn(\epsilon)\Big),
\end{aligned}\]
where noticing that $n_{r}(\zeta>r)=1$ and $n(\zeta>r) n_{r}=n(\cdot\mathbf{1}(\zeta>r))$. Let $r\to0+$ gives
\begin{equation}\label{eqn:A2}
\begin{aligned}
&\ \E\bigg[\exp\Big(-\sum_{s\le t}\eta^{(p,q)}(\epsilon_{s},\zeta_{s})\Big)\mathbf{1}\big(\zeta_{s}<\kappa_{b,s}^{+}\wedge\kappa_{c,s}^{-},\forall s\le t\big)\bigg]\\
=&\ \exp\Big(-t \cdot n\big(1-e^{-(p\zeta^{+}+q\zeta^{-})}\mathbf{1}(\zeta<\kappa_{b}^{+}\wedge\kappa_{c}^{-})\big)\Big),
\end{aligned}
\end{equation}
where $\eta^{(p,q)}(\epsilon,\zeta)=p\zeta^{+}+q\zeta^{-}$.
Plugging \eqref{eqn:A2} into \eqref{eqn:A1} and applying \eqref{lapocc:bc} gives \eqref{n:pq:bc}.

The proofs for \eqref{n:pq:b<c}, \eqref{n:pq:c<b} and \eqref{n:pq:resl} are similar to those for \eqref{eqn:n:b<cr}, \eqref{eqn:n:c<br} and \eqref{resl:np}.

On sets $\{\tau_{b}^{+}<\infty\}$ and $\{\tau_{c}^{-}<\infty\}$ we have respectively,
\[\begin{aligned}
&\ \exp\Big(-\mo(\tau_{b}^{+})\Big)\mathbf{1}\big(\tau_{b}^{+}<\tau_{c}^{-}\big)\\
=&\ \sum_{t}e^{-p\kappa_{b,t}^{+}}\mathbf{1}\big(\kappa_{b,t}^{+}<\zeta_{t}\big)
\times\exp\Big(-\sum_{s<t}\eta^{(p,q)}(\epsilon_{s},\zeta_{s})\Big)\mathbf{1}\big(\zeta_{s}<\kappa_{b,s}^{+}\wedge\kappa_{c,s}^{-},\forall s<t\big),\\
&\ \exp\Big(-\mo(\tau_{c}^{-})\Big)\mathbf{1}\big(\tau_{c}^{-}<\tau_{b}^{+}\big)\\
=&\ \sum_{t}e^{-p\zeta_{t}^{+}-q(\kappa_{c,t}^{-}-\zeta^{+})}\mathbf{1}\big(\kappa_{c,t}^{-}<\kappa_{b,t}^{+}\wedge\zeta_{t}\big)
\times\exp\Big(-\sum_{s<t}\eta^{(p,q)}(\epsilon_{s},\zeta_{s})\Big)\mathbf{1}\big(\zeta_{s}<\kappa_{b,s}^{+}\wedge\kappa_{c,s}^{-}, \forall s<t\big),
\end{aligned}\]
where $\eta^{(p,q)}(\kappa_{c,t}^{-})=p\zeta_{t}^{+}+q(\kappa_{c,t}^{-}-\zeta^{+})$ for $\kappa_{c,t}^{-}<\infty$ due to the absence of positive jumps for the excursion $\epsilon_{t}$.
Applying the compensation formula and \eqref{eqn:A2} we have
\[\begin{aligned}
\E\Big[\exp\Big(-\mo(\tau_{b}^{+})\Big); \tau_{b}^{+}\le\tau_{c}^{-}\Big]=&\ \frac{n\big(e^{-p\kappa_{b}^{+}}; \kappa_{b}^{+}<\zeta)}
{n\big(1-e^{-(p\zeta^{+}+q\zeta^{-})}\mathbf{1}(\zeta<\kappa_{b}^{+}\wedge\kappa_{c}^{-})\big)},\\
\E\Big[\exp\Big(-\mo(\tau_{c}^{-})\Big); \tau_{c}^{-}\le\tau_{b}^{+}\Big]=&\ \frac{n\big(e^{(q-p)\zeta^{+}-q\kappa_{c}^{-}};\kappa_{c}^{-}<\kappa_{b}^{+}\wedge\zeta\big)}{n\big(1-e^{-(p\zeta^{+}+q\zeta^{-})}\mathbf{1}(\zeta<\kappa_{b}^{+}\wedge\kappa_{c}^{-})\big)}.
\end{aligned}\]
Applying \eqref{lapocc:b<c} and \eqref{lapocc:c<b},
one can prove \eqref{n:pq:b<c} and \eqref{n:pq:c<b}, respectively.

Lastly, let $e_{p}$ be an exponential random variable independent of $X$, we have
\[\begin{aligned}
&\ \E\Big[\exp\Big((p-q)\mom(e_{p})\Big)\cdot f\big(X(e_{p})\big); e_{p}<\tau_{b}^{+}\wedge\tau_{c}^{-}\Big]\\
=&\ \E\bigg[\sum_{t}\mathbf{1}\big(L^{-1}(t-)<e_{p}<L^{-1}(t)\big) \cdot\mathbf{1}\big(e_{p}-L^{-1}(t-)<\zeta_{t}\wedge\kappa_{b,t}^{+}\wedge\kappa_{c,t}^{-}\big)\\
&\quad\times f\big(\epsilon_{t}(e_{p}-L^{-1}(t-))\big)\cdot\exp\Big((p-q)\int_{0}^{e_{p}-L^{-1}(t-)}\mathbf{1}\big(\epsilon_{t}(s)<0)\big)\,ds\Big)\\
&\quad\times \exp\Big((p-q)\sum_{s<t}\zeta_{s}^{-}\Big)\cdot
\mathbf{1}\big(\zeta_{s}<\infty,\zeta_{s}<\kappa_{b,s}^{+}\wedge\kappa_{c,s}^{-}, \forall s<t\big)\bigg].
\end{aligned}\]
By the memoryless property of $e_{p}$ on $\{e_{p}>L^{-1}(t-)\}$ and the fact $L^{-1}(t-)=\sum_{s<t}\zeta_{s}$ from \eqref{d=0}, we have
\[\begin{aligned}
&\ p \int_{0}^{\infty}\E\Big[e^{-\mo(t)}f\big(X(t)\big); t<\tau_{b}^{+}\wedge\tau_{c}^{-}\Big]\,dt\\
=&\ \E\bigg[\sum_{t} f\big(\epsilon_{t}(e_{p})\big)\cdot\exp\Big((p-q)\int_{0}^{e_{p}}\mathbf{1}\big(\epsilon_{t}(s)<0)\big)\,ds\Big)\cdot\mathbf{1}\big(e_{p}<\zeta_{t}\wedge\kappa_{b,t}^{+}\wedge\kappa_{c,t}^{-}\big)\\
&\quad\times \exp\Big((p-q)\sum_{s<t}\zeta_{s}^{-}-pL^{-1}(t-)\Big)\cdot
\mathbf{1}\big(\zeta_{s}<\infty,\zeta_{s}<\kappa_{b,s}^{-}\wedge\kappa_{c,s}^{-}, \forall s<t\big)\bigg]\\
=&\ p\cdot \E\bigg[\sum_{t}\Big(\int_{0}^{\zeta_{t}\wedge\kappa_{b,t}^{+}\wedge\kappa_{c,t}^{-}}e^{-\eta^{(p,q)}(s)}f\big(\epsilon_{t}(s)\big)\,ds\Big)\cdot e^{-\sum_{s<t}\eta^{(p,q)}(\epsilon_{s},\zeta_{s})}\mathbf{1}\big(\zeta_{s}<\kappa_{b,s}^{+}\wedge\kappa_{c,s}^{-},\forall s<t\big)\bigg].
\end{aligned}\]
Applying the compensation formula, \eqref{lapocc:resl} and \eqref{eqn:A2} we have \eqref{n:pq:resl}. This completes the proof.
\end{proof}


\end{document}